\documentclass[reqno,11pt]{amsart}
\usepackage{amsmath,amssymb,amsthm, verbatim}
\usepackage{url}
\usepackage[usenames, dvipsnames]{color}

\usepackage[letterpaper,hmargin=1in,vmargin=1in]{geometry}

\usepackage{graphicx}
\usepackage{enumerate,pinlabel}
\usepackage{mathrsfs,graphicx,url}
\usepackage[usenames,dvipsnames]{xcolor}
\usepackage[colorlinks=true,linkcolor=Black,citecolor=Black, urlcolor=Black]{hyperref}

\numberwithin{equation}{section}

\theoremstyle{plain}
\newtheorem{theorem}{Theorem}
\newtheorem*{theorem*}{Theorem}
\newtheorem*{lemma*}{Lemma}
\newtheorem{lemma}{Lemma}[section]
\newtheorem{proposition}[lemma]{Proposition}

\theoremstyle{definition}

\theoremstyle{remark}

\newcommand{\ep}{\varepsilon}

\newcommand{\R}{\mathbb{R}}
\newcommand{\US}{\mathbb{S}}

\newcommand\supp{\mathop{\rm supp}}

\newcommand\real{\mathop{\rm Re}}
\newcommand\imag{\mathop{\rm Im}}
\newcommand*{\defeq}{\mathrel{\vcenter{\baselineskip0.5ex \lineskiplimit0pt

                     \hbox{\scriptsize.}\hbox{\scriptsize.}}}%
                     =}

\begin{document}
\title{Semiclassical resolvent bounds for weakly decaying potentials}

\author{Jeffrey Galkowski}
\address{Department of Mathematics, University College London, London, UK}
\email{j.galkowski@ucl.ac.uk}

\author{Jacob Shapiro}
\address{Mathematical Sciences Institute, Australian National University, Acton, ACT, Australia}
\email{Jacob.Shapiro@anu.edu.au}

\thanks{This material is based upon work supported by the National Science Foundation under Grant No. 1440140, while the authors were in residence at the Mathematical Sciences Research Institute in Berkeley, California, during the Fall semester of 2019. J. Shapiro was also supported in part by the Australian Research Council through grant DP180100589}

\begin{abstract}
 In this note, we prove weighted resolvent estimates for the semiclassical Schr\"odinger operator $-h^2 \Delta + V(x) : L^2(\R^n) \to L^2(\R^n)$, $n \neq 2$. The potential $V$ is real-valued, and assumed to either decay at infinity or to obey a radial $\alpha$-H\"older continuity condition, $0\leq \alpha \leq 1$, with sufficient decay of the local radial $C^\alpha$ norm toward infinity. Note, however, that in the H\"older case, the potential need \emph{not} decay. If the dimension $n \ge 3$, the resolvent bound is of the form $\exp \left(C h^{-1 - \frac{1 - \alpha}{3 + \alpha}} [(1-\alpha) \log(h^{-1})+c]\right)$, while for $n = 1$ it is of the form $\exp(Ch^{-1})$. A new type of weight and phase function construction allows us to reduce the necessary decay even in the pure $L^\infty$ case.
\end{abstract}
\maketitle

\section{Introduction and statement of results}
Let $\Delta \defeq \sum_{j =1}^n \partial^2_j \le 0$ be the Laplacian on $\mathbb{R}^n$, $n \neq 2$. In this article, we study the semiclassical Schr\"odinger operator
\begin{equation*}
 P(h) \defeq -h^2 \Delta + V : L^2(\mathbb{R}^n) \to L^2(\mathbb{R}^n),\qquad h > 0,
\end{equation*}
where $V \in L^\infty(\R^n ; \R)$.  We assume either that $V$ satisfies a radial $\alpha$-H\"older continuity condition, $0\leq \alpha \leq 1$, or that it is only $L^\infty$ but decaying. When $n \ge 3$,  we use\\$(r, \theta) = (|x|, x/|x|) \in (0, \infty) \times \US^{n-1}$ to denote polar coordinates on $\R^n \setminus \{0\}$. 

When $V$ is only $L^\infty$, we assume 
\begin{equation}
\label{e:Linfty}
|V|\leq c_1\langle r\rangle^{-2}m(r),
\end{equation}
for some 
\begin{equation}
\label{e:m}
c_1 > 0, \qquad 0< m(r)\leq 1,\qquad m(r)\langle r\rangle^{-1/2}\in L^2(0,\infty),
\end{equation}
and where $\langle x \rangle = \langle r \rangle \defeq (1 + r^2)^{1/2}$.

Since $V \in L^\infty(\R^n; \R)$, by the Kato-Rellich Theorem, $P(h)$ is self-adjoint $L^2(\R^n) \to L^2(\R^n)$ with respect to the domain $H^2(\mathbb{R}^n)$. Therefore, the resolvent $(P - z)^{-1}$ is bounded $L^2(\R^n) \to L^2(\R^n)$ for all $z \in \mathbb{C} \setminus \mathbb{R}$, and we obtain
\begin{theorem}
\label{thm Linfty high d}
Let $n \ge 3$, $m$ as in~\eqref{e:m}, $c_1>0$ and $E>0$. Then there are $C>0$ and $h_0\in (0,1]$ so that for all $s> 1/2,$ there is $C_s>0$ such that for all $V \in L^\infty(\R^n;\R)$ satisfying~\eqref{e:Linfty},
\begin{equation}
\label{e:mainResult Linfty}
g^{\pm}_s(h,\ep)\leq C_s \exp\left( h^{-\frac{4}{3}}(C \log h^{-1}+C_s) \right),\qquad \varepsilon>0,\,h\in(0,h_0],
\end{equation}
where
\begin{equation} \label{defn g}
    g^\pm_{s}(h, \varepsilon) \defeq \|\langle x \rangle^{-s} (P(h) - E \pm i\varepsilon)^{-1}  \langle x \rangle^{-s}\|_{L^2(\R^n) \to L^2(\R^n)}.
\end{equation}
\end{theorem}
\bigskip
When $V$ has some radial $\alpha$-H\"older regularity,  $0\leq\alpha\leq 1$, we need not assume that $V$ decays towards infinity. Instead, we suppose
\begin{equation}
\label{Holder short range}
V\in L^\infty,\qquad
\limsup_{y\to 0^+}\sup_{r}\frac{|V(r\theta)-V((r+y)\theta)|}{|y|^\alpha}\langle r\rangle^{3}m^{-2}(r)\leq c_2,\qquad \theta\in \US^{n-1},
\end{equation}
for some $c_2 > 0$. We also define
\begin{gather}
V_\infty \defeq \limsup_{r\to \infty}\sup_{\theta\in \US^{n-1}}V(r\theta),  \label{defn E 0}\\
0 < \delta_V \defeq \inf\Big\{ y > 0  \mid \sup_{r}\frac{|V(r\theta)-V((r+y)\theta)|}{|y|^\alpha}\langle r\rangle^{3}m^{-2}(r)> 2c_2\Big\},\label{defn eps0} 
\end{gather}
and for $E>V_\infty$,
\begin{gather}
 R_{E,V} \defeq \sup\Big\{r \mid \sup_{\theta \in \US^{n-1}} V(r\theta) > \frac{E + 3 V_\infty}{4}\Big\} \label{e:R_E}. 
\end{gather}
\smallskip

\noindent{\bf{Remark:}} Note that when $\alpha=0$ and~\eqref{Holder short range} holds, $V$ is still only $L^\infty$, but the magnitude of its fluctuations are decaying faster than those in~\eqref{e:Linfty}. 
\smallskip

In this H\"older regular case, we obtain
\begin{theorem}
\label{thm higher dim}
Let $n \ge 3$, $m$ as in~\eqref{e:m}, $c_2>0$, $R_E>0$, $C_V\in \mathbb{R}$, ${E_\infty}\in \mathbb{R}$, and $E>E_\infty$. Then there is $C>0$ such that for all $\delta_1>0$, there is $h_0\in (0,1]$ so that for all $s> 1/2$, there is $C_s>0$ so that for $V \in L^\infty(\R^n;\R)$ obeying $\sup_{\mathbb{R}^n}V\leq C_V$, $V_\infty\leq E_{\infty}$, $\delta_1 \le \delta_V$, $R_{E,V}\leq R_E$, and \eqref{Holder short range} for some $0\leq\alpha \leq 1$,
\begin{equation}
\label{e:mainResult}
g^{\pm}_s(h,\ep)\leq C_s \exp\left( h^{-1-\sigma_\alpha}(C\sigma_\alpha \log h^{-1}+C_s) \right),\qquad \varepsilon>0,\,h\in(0,h_0],
\end{equation}
where
$$
\sigma_\alpha:=\frac{1-\alpha}{3+\alpha}.
$$
\end{theorem}

\smallskip

In the one-dimensional case,~\eqref{Holder short range} can be relaxed further to
\begin{equation} \label{Holder condition oned}
\begin{gathered}
\limsup_{y\to 0}\sup_x \frac{|V(x) - V(x+y)|}{m_0(|x|)} \le  c_0,
\end{gathered}
\end{equation}
for some
\begin{equation}
\label{e:m1}
c_0 > 0, \qquad 0<m_0(r) \leq 1,\qquad m_0\in L^1(0,\infty).
\end{equation}
We then define
\begin{equation} \label{defn delta0}
0 < \delta_{0,V} \defeq  \inf \{ y > 0 \mid \sup_x \frac{|V(x) - V(x  + y)|}{m_0(|x|)} > 2c_0 \}.
\end{equation}
Then we have the following one dimensional result.
\begin{theorem} \label{thm oned}
Let $n =1$, $m_0$ as in~\eqref{e:m1}, $c_0>0$, $R_E>0$, $C_V,E_\infty\in \mathbb{R}$ and $E>E_\infty$. Then there is $C>0$ such that for all $\delta_0 >0$, there is $h_0 \in (0,1]$ so that for all $s> 1/2$, there is $C_s>0$ so that for $V \in L^\infty(\R; \R)$ obeying  $ \delta_0 \le \delta_{0,V}$, $\sup_{\mathbb{R}} V \leq C_V$, $R_{E,V}\leq R_E$, and \eqref{Holder condition oned},
\begin{equation} \label{resolv est oned}
 g_{s}^\pm(h, \varepsilon)  \le C_s \exp \left( C h^{-1} \right), \qquad \varepsilon > 0, \, h \in (0,h_0].
\end{equation}
\end{theorem} 

Bounds on $g^{\pm}_s$ are known to hold under various geometric, regularity, and decay assumptions. Burq \cite{bu98, bu02} showed  $g^{\pm}_s \le e^{Ch^{-1}}$ for $V$ smooth and decaying sufficiently fast near infinity, and also for more general perturbations of the Laplacian. Cardoso and Vodev \cite{cavo} extended Burq's estimate to infinite volume Riemannian manifolds which may contain cusps. 

In lower regularity and $n \neq 2$, Datchev \cite{da14} showed $g^{\pm}_s \le e^{Ch^{-1}}$, provided $V, \partial_r V \in L^\infty(\R^n; \R)$ and have long-range decay. The second author \cite{sh19} obtained the same bound for $n =2$, and under the same assumptions, except with $\partial_r V$ replaced by $\nabla V$ \cite{sh19}. On the other hand, Vodev \cite{vo14} showed that, if $n \ge 3$ and $V$'s radial $\alpha$-H\"older moduli are $O(h^\nu \langle r \rangle^{-\kappa})$, where $\nu > 0 $, $\kappa > 1$, and $\alpha \ge 1 - 2\nu$, then $g^{\pm}_s \le e^{Ch^{-\ell}}$, where
\begin{equation*}
\ell = \max \left\{0, \frac{2(1 - \nu - \alpha)}{1 - \alpha} \right\} < 1.
\end{equation*}

If $V \in L_{\text{comp}}^\infty(\R^n; \R)$, $n \ge 2$, it was previously shown \cite{klvo19,sh17} that $g^{\pm}_s \le e^{Ch^{-4/3}\log(h^{-1})}$. This same bound was extended to short range potentials on $\R^n$ \cite{vo19a, vo19b}, and then to short range potentials on a large class of asymptotically Euclidean manifolds \cite{vo19c}. If $n =1$, $g^{\pm}_s \le e^{Ch^{-1}}$, even if $V \in L^1(\R;\R)$ \cite{dash19}.

Theorems~\ref{thm Linfty high d} and~\ref{thm higher dim} improve upon the existing literature in several ways. First, in the pure $L^\infty$ case~\eqref{e:Linfty}, Theorem~\ref{thm Linfty high d} reduces the required decay for $V$ from that in~\cite{vo19a,vo19b}. While we are still unable to obtain estimates when $V$ is an arbitrary short range $L^\infty$ potential without an additional loss of powers of $h$ in $\log (g_s^{\pm}(h,\ep))$, the decay  assumed in~\eqref{e:Linfty} appears to improve on the existing literature by one order in $r$. Secondly, the assumptions for Theorem~\ref{thm higher dim}~\eqref{Holder short range} allow for \emph{non-decaying} potentials provided some control on the local oscillations of the potential $V$ (even if $V$ is not H\"older continuous for any positive $\alpha$). Finally, as the H\"older constant of the potential varies between $0$ and $1$, the results interpolate between those in the $L^\infty$ and Lipschitz cases, with the bound on $g_s^{\pm}(h,\ep)$ agreeing with the existing estimates at both endpoints.

Next, Theorem~\ref{thm oned} seems to be the first semiclassical resolvent estimate in one dimension that does not require $V$ or  $\partial_xV$ to belong to $L^1(\R;\R)$. Again, by imposing some condition on the oscillations of $V$, we are able to handle even non-decaying potentials.

In dimension $n \ge 2$, it is an open problem to determine the optimal $h$-dependence of the resolvent for $V \in L^\infty$ or $V$ satisfying \eqref{Holder short range}. In contrast, it is well known that the bound $e^{Ch^{-1}}$ cannot be improved in general. See, for instance, \cite{ddz15} and the references cited there.

To prove Theorems~\ref{thm Linfty high d},~\ref{thm higher dim} and~\ref{thm oned}, we adapt the Carleman estimates proved in \cite{vo19a} and \cite{dash19}. In addition to the modifications necessary to take advantage of the H\"older regularity of $V$, the main improvement in our argument is to determine $\varphi$ and $w$ from the logarithmic derivatives of respectively $\varphi'$ and $w$. This dramatically simplifies the computations necessary to construct the requisite phases and weights. See~\eqref{e:defPhiW} and~\eqref{e:keyCalc} for the main quantities one must estimate.

In the final stages of writing this note, we learned of the article~\cite{vo20}, in which Vodev uses a somewhat different weight and phase construction to study H\"older potentials analogous to ours. However, the assumed decay in that article is stronger than what we need here. On the other hand, Vodev's article gives the local Carleman estimates necessary to handle dimension $n=2$ as well as the case where $\mathbb{R}^n$ is replaced by the exterior of a smooth obstacle.

\section{Preliminary Calculations and Lemmata}
As in most previous proofs of resolvent estimates for low regularity potentials, the backbone of the proof is a Carleman estimate. We start from the identity 
\begin{equation*}
    r^{\frac{n-1}{2}}(- \Delta) r^{-\frac{n-1}{2}} = -\partial^2_r + \Lambda,
\end{equation*}
where 
\begin{equation} \label{Lambda positive}
    \Lambda \defeq \frac{1}{r^2}\left( -\Delta_{\US^{n-1}} + \frac{(n-1)(n-3)}{4} \right) \ge 0,
\end{equation}
and $\Delta_{\US^{n-1}}$ denotes the negative Laplace-Beltrami operator on $\US^{n-1}$. Then,
we form the conjugated operator
\begin{equation} \label{conjugation}
\begin{split}
  P^{\pm}_\varphi(h) &\defeq e^{\varphi/h} r^{\frac{n-1}{2}}\left( P(h) - E \pm i\varepsilon \right) r^{-\frac{n-1}{2}} e^{-\varphi/h}\\
  &= -h^2\partial^2_r + 2h \varphi' \partial_r + h^2\Lambda + V -(\varphi')^2 + h\varphi''  - E \pm i\varepsilon.
 \end{split}
\end{equation}

Now, let $V_h\in C^\infty((0,\infty)_r;L^\infty(\US_\theta^{n-1}))$ be a smoothed approximation to $V$, and define 
\begin{equation}
\label{e:remainder}
R_h:=V-V_h.
\end{equation}
For $n\geq 3$ and $u \in e^{\varphi/h} r^{(n-1)/2} C^\infty_{\text{comp}}(\R^n)$, we define a spherical energy functional $F[u](r)$,
\begin{equation} \label{F high dim}
    F(r) = F[u](r) \defeq \|hu'(r, \cdot)\|^2 - \langle (h^2\Lambda + V_h -(\varphi')^2- E)u(r, \cdot), u(r, \cdot) \rangle,
\end{equation}
where $\| \cdot \|$ and $\langle \cdot, \cdot \rangle$ denote the norm and inner product on $L^2(\mathbb{S}_\theta^{n-1})$, respectively.  The derivative of $F$, in the sense of distributions on $(0,\infty)$, is
\begin{equation*}
\begin{split}
    F' &= 2  \real \langle h^2 u'', u'\rangle -2 \real \langle (h^2\Lambda + V_h - E)u, u' \rangle + 2r^{-1} \langle h^2 \Lambda u, u \rangle - ((\varphi')^2 - V_h)' \| u\|^2  \\
    &= -2 \real \langle P^{\pm}_\varphi(h) u, u' \rangle + 2r^{-1} \langle h^2 \Lambda u, u \rangle + ((\varphi')^2-V_h)'\|u\|^2 + 4h^{-1} \varphi' \|hu'\|^2 \\
    &\qquad\mp 2\varepsilon \imag \langle u,u'\rangle + 2\real \langle (R_h + h \varphi'') u, u' \rangle. 
    \end{split}
\end{equation*}
Thus $(wF)'$, as a distribution on $(0,\infty)$, is given by 
\begin{equation} \label{deriv wF}
\begin{split}
    (wF)' &= w'F + wF' \\
    &= w'\|hu'\|^2 - w'\langle (h^2\Lambda + V_h- (\varphi')^2- E)u, u \rangle \\
    & -2 w \real \langle P^{\pm}_\varphi(h) u, u' \rangle + 2wr^{-1} \langle h^2 \Lambda u, u \rangle + w((\varphi')^2 - V_h)' \| u\|^2 + 4h^{-1} w \varphi \|hu'\|^2 \\
    &\qquad \mp 2\varepsilon w \imag \langle u,u'\rangle + 2\real \langle (R_h + h \varphi'') u, u' \rangle \\
    &= -2 \real w \langle P^{\pm}_\varphi(h) u, u' \rangle \mp 2\varepsilon w \imag \langle u,u'\rangle + (2wr^{-1} - w') \langle h^2\Lambda u,u\rangle \\
    &\qquad+ (4h^{-1}w \varphi' + w')\|hu'\|^2  + (w(E+ (\varphi')^2-V_h ))' \|u\|^2 + 2w\real \langle (R_h + h \varphi'') u, u' \rangle. 
    \end{split}
\end{equation}
Using \eqref{Lambda positive}  when $n\geq 3$, we will need
\begin{equation}
\label{e:weightCond}
2wr^{-1}-w'\geq 0,
\end{equation} 
to control the term involving $\Lambda$. It is the absence of this condition which allows for the improved estimate in dimension one. Using~\eqref{e:weightCond} together with $2ab \ge -(\gamma a^2 + \gamma^{-1}b^2)$ for all $\gamma > 0$, we find
\begin{equation}
\label{e:lowerDerivative}
\begin{aligned}
    w' F + w F' &\ge -\frac{3 w^2}{h^2w'} \| P^{\pm}_\varphi(h)u \|^2 \mp  2\varepsilon w \imag \langle u,u'\rangle + \frac{1}{3}(w' + 4h^{-1} \varphi' w)\|hu' \|^2 \\
    &\qquad+ (w(E + (\varphi')^2-V_h))' \|u\|^2 -\frac{3(w(h^{-1}|R_h| + \varphi''))^2}{w' + 4h^{-1}\varphi' w} \|u\|^2.
    \end{aligned}
\end{equation}

In dimension $n=1$, rather than the spherical energy~\eqref{F high dim}, we use the pointwise energy
\begin{equation*}
    F(x) = F[u](x) \defeq |hu'(x)|^2 -  (V_h(x) -(\varphi'(x))^2- E)|u(x)|^2.
\end{equation*}
Exactly the same computations then lead to 
\begin{equation*}
\begin{aligned}
    w' F + w F' &\ge -\frac{3 w^2}{h^2w'} | P^{\pm}_\varphi(h)u |^2 \mp  2\varepsilon w \imag  u\overline{u'} + \frac{1}{3}(w' + 4h^{-1} \varphi' w)|hu' |^2 \\
    &\qquad+ (w(E+ (\varphi')^2-V_h ))'|u|^2 -\frac{3(w(h^{-1}|R_h| + \varphi''))^2}{w' + 4h^{-1}\varphi' w} |u|^2.
    \end{aligned}
\end{equation*}

Thus, the main goal of the estimates below will be to construct $\varphi$ and $w$ such that 
$$
(w(E + (\varphi')^2-V_h))'-\frac{3(w(h^{-1}|R_h| + \varphi''))^2}{w' + 4h^{-1}\varphi' w} \geq \frac{E - E_\infty}{2}w'.
$$
Putting 
$$
A(r):=(w(E+ (\varphi')^2-V_h))',\qquad B(r):=\frac{(w(h^{-1}|R_h| + \varphi''))^2}{w' + 4h^{-1}\varphi' w},
$$
our goal is thus, for $K > 0$ fixed and $h$ small enough, to find $w$ and $\varphi$ such that 
\begin{equation}
\label{e:goalEst}
A(r)-\frac{K}{2}B(r) \geq \frac{E - E_\infty}{2}w'(r).
\end{equation}
Now, we will assume throughout that $w',\varphi'>0$. Therefore, putting 
\begin{equation}\label{e:defPhiW}
\Phi:=\frac{\varphi''}{\varphi'}=(\log \varphi')',\qquad \mathcal{W}:=\frac{w}{w'}=\frac{1}{(\log w)'},
\end{equation}
we calculate
\begin{align*}
A(r)-\frac{K}{2}B(r)&=w'(E+(\varphi')^2-V_h)+w(2\varphi'\varphi''-V_h')-\frac{K}{2}\frac{(w(h^{-1}|R_h| + \varphi''))^2}{w' + 4h^{-1}\varphi' w}\\
&=w'\Big[E+(\varphi')^2-V_h+\mathcal{W}(2\varphi'\varphi''-V_h')-\frac{K}{2}\frac{(w(h^{-1}|R_h| + \varphi''))^2}{w'^2 + 4h^{-1}\varphi' ww'}\Big]\\
&= w'\Big[ E+(\varphi')^2(1+2\mathcal{W}\Phi)-V_h-\mathcal{W}V_h'-\frac{K}{2}\mathcal{W}^2\frac{((h^{-1}|R_h| + \varphi''))^2}{1 + 4h^{-1}\varphi' \mathcal{W}}\Big]\\
&\geq w'\Big[ E+(\varphi')^2(1+2\mathcal{W}\Phi)-V_h-\mathcal{W}V_h'-K\mathcal{W}^2\frac{h^{-2}|R_h|^2 + (\varphi'')^2}{1 + 4h^{-1}\varphi' \mathcal{W}}\Big].
\end{align*}
Finally,
\begin{equation}
\label{e:keyCalc}
\begin{aligned}
A(r)-\frac{K}{2}B(r)&\geq w'\Big[ E+(\varphi')^2(1+2\mathcal{W}\Phi-K\mathcal{W}\Phi^2\min(\mathcal{W},\tfrac{h}{4\varphi'}))\\
&\qquad\qquad-V_h-\mathcal{W}V_h'-K\mathcal{W}h^{-2}|R_h|^2\min(\mathcal{W},\tfrac{h}{4\varphi'}) \Big].
\end{aligned}
\end{equation}
The key improvement in this article is that, to prove the main estimates \eqref{e:keyLower} and \eqref{e:keyLower1}, we work with $\mathcal{W}$ and $\Phi$ rather than directly with $w$ and $\varphi$. This simplifies the calculations dramatically and points the way to a new choice of phase function allowing us to weaken the decay requirements on $V$. The condition~\eqref{e:weightCond} for $n\geq 3$ translates simply to $\Phi\geq r/2$. The remainder of the article focuses on constructing appropriate $\mathcal{W}$ and $\Phi$ such that~\eqref{e:goalEst} holds.

Before proceeding with the construction of $\mathcal{W}$ and $\Phi$, we need a few elementary lemmata:
\begin{lemma}
\label{l:phi}
Let 
$$
\Phi(s)=-\frac{1}{s+1+\Phi_1(s)},
$$
with 
\begin{equation} \label{Phi 1 conditions}
0 \leq (s+1)^{-2}\Phi_1(s)\in L^1(0,\infty).
\end{equation}
Then,
$$
-\log(r+1)\leq \int_0^r \Phi(s)ds\leq  -\log(r+1)+\|(s+1)^{-2}\Phi_1(s)\|_{L^1(0,\infty)}.
$$
\end{lemma}
\begin{proof}
First, note that 
\begin{align*}
\log (r+1)+\int_0^r \Phi(s)ds&=\int_0^r \frac{1}{1+s}-\frac{1}{s+1+\Phi_1(s)}ds\\
&=\int_0^r\frac{\Phi_1(s)}{(s+1)(s+1+\Phi_1(s))}ds.
\end{align*}
Next, note that 
$$
0\leq \int_0^r\frac{\Phi_1(s)}{(s+1)(s+1+\Phi_1(s))}ds\leq \|(s+1)^{-2}\Phi_1(s)\|_{L^1(0,\infty)},
$$
which implies
$$
-\log (r+1)\leq \int_0^r\Phi(s)ds\leq -\log(r+1)+\|(s+1)^{-2}\Phi_1(s)\|_{L^1(0,\infty)}.
$$
\end{proof}

In the proof of Theorem \ref{thm higher dim}, we will need to approximate $V$ by smooth functions $V_h$. In the case~\eqref{e:Linfty}, we simply approximate $V$ by $0$, defining $V_h\equiv 0$. On the other hand, when we assume~\eqref{Holder short range}, we make a non-trivial approximation to $V$. In the spirit of \cite[Section 2]{vo14}, let 
\begin{equation}\label{e:chi}
\chi\in C_{\text{comp}}^\infty((0,1);[0,1]),\qquad\qquad \int \chi(s)ds=1,
\end{equation}
and define
$$
V(r\theta;\gamma) \defeq \int^\infty_0 V((r +  \gamma s) \theta ) \chi (s) ds = \gamma^{-1} \int_0^\infty V(s\theta) \chi (\gamma^{-1}(s - r)) d s, \qquad 0< \gamma \le  1.
$$ 
Then set 
$$
V_h(r\theta) \defeq V(r\theta;h^\rho),
$$
for $\rho > 0$ to be chosen later, depending on $\alpha$.
\begin{lemma}
\label{l:Vest}
Suppose $0 \le \alpha \le 1$, $V$ satisfies \eqref{Holder short range}, and $\delta_V$ is as in \eqref{defn eps0}. Then there exists $C_\chi> 0$ depending only on $\chi$ so that, for all $h\in (0,\delta_V^{1/\rho}]$,
\begin{equation}
\begin{gathered} \label{e:smoothedVbounds}
  V_h(r\theta)\leq \sup_{s \in[r,r+h^\rho]}V(s\theta),\\
 |V_h'(r\theta)|\leq C_\chi c_2h^{-\rho(1-\alpha)}\langle r\rangle^{-3}m^2(r),\qquad |R_h(r\theta)|\leq  c_2h^{\rho\alpha}\langle r\rangle^{-3}m^2(r).
\end{gathered}
\end{equation}
\end{lemma}
\begin{proof}
First observe that
\begin{equation} \label{sup Vh}
\begin{aligned}
V(r\theta;\gamma)&= \int^\infty_0 [V((r +  \gamma s) \theta )-\inf_{t\in [r,r+\gamma]}V(t\theta)] \chi (s) ds +\inf_{t\in [r,r+\gamma]}V(t\theta)\\
 &\le (\sup_{s\in[r,r+\gamma]} V(s\theta)-\inf_{t\in [r,r+\gamma]}V(t\theta)) \int \chi(s)ds +\inf_{t\in [r,r+\gamma]}V(t\theta)\\
 & = \sup_{s\in[r,r+\gamma]} V(s\theta)
 \end{aligned}
\end{equation}
where in the third line we use implicitly that $\chi\geq 0$ and for $s\in \supp \chi$,\\  $[V((r +  \gamma s) \theta )-\inf_{t\in [r,r+\gamma]}V(t\theta)]\geq 0$.

Next, from $\int \chi'dr=0$, 
\begin{align*}
|V'(r\theta;\gamma)|&=\left|\gamma^{-2}\int_0^\infty V(s\theta)\chi'(\gamma^{-1}(s-r))ds-\gamma^{-1}V(r\theta)\int_0^1\chi'(s)ds\right|\\
&=\left|\gamma^{-1}\int_0^1 [V((r+\gamma s)\theta)-V(r\theta)]\chi'(s)ds\right|\\
&\leq\left|\gamma^{-1+\alpha}\int_0^1 s^{\alpha}\frac{ (V((r+\gamma s)\theta)-V(r\theta))\chi'(s)}{\gamma^\alpha s^\alpha}ds\right|.
\end{align*}
In particular, by~\eqref{Holder short range} and the definition~\eqref{defn eps0} of $\delta_V$, we have, for $0 < \gamma \leq \delta_V$,
\begin{equation} \label{bound Vh prime}
|V'(r\theta;\gamma)|\leq 2c_2 \gamma^{-1+\alpha}\langle r\rangle^{-3}m^2(r)\int_0^1 |s^{\alpha}\chi'(s)|ds \le  C_\chi c_2 \gamma^{-1+\alpha}\langle r\rangle^{-3}m^2(r).
\end{equation}
Finally, using~\eqref{Holder short range} again,
\begin{equation} \label{bound Rh}
\begin{split}
|V(r\theta)-V(r\theta;\gamma)|&=\left| \int_0^\infty [V(r\theta)-V((r+\gamma s)\theta)]\chi(s)ds\right|\\
&=\left| \int_0^\infty\gamma^\alpha s^\alpha \frac{V(r\theta)-V((r+\gamma s)\theta)}{\gamma^\alpha s^\alpha} \chi(s)ds\right|\\
&\leq c_2 \gamma^\alpha \langle r\rangle^{-3}m^2(r),
\end{split}
\end{equation}
for $0 < \gamma \le \delta_V$. The lemma is proved by setting $\gamma=h^\rho$, $h \in (0,\delta_V^{1/\rho}]$, in \eqref{sup Vh}, \eqref{bound Vh prime}, and \eqref{bound Rh}. \\
\end{proof}

\section{Proof of the main estimates ($n\geq 3$)}
Recall the definitions of $\Phi$ and $\mathcal{W}$ from~\eqref{e:defPhiW}, and put
\begin{equation}
\label{e:IC}
\varphi(r)=h^{-\sigma}\varphi_0(r), \, \sigma \ge 0, \qquad \varphi_0(0)=0,\,\varphi_0'(0)=\tau_0 \ge 1,\qquad w(0)=0, \,w'(0)=1,
\end{equation}
so that 
\begin{equation} \label{e:solve for w and phi}
\Phi=(\log \varphi_0')',\qquad \mathcal{W}=\frac{1}{(\log w)'}.
\end{equation}
We also set
\begin{equation}
\label{e:sigma}
\sigma=\frac{1-\alpha}{3+\alpha},\qquad \rho=\frac{2}{3+\alpha}.
\end{equation}
Finally, let
\begin{equation} \label{defn a}
a = a_0h^{-M}, \, a_0 \ge 1, \, M > 0.
\end{equation}
Each of the parameters $\sigma$, $\tau_0$, $a_0$, and $M$ will be fixed shortly.

The main result of this section is Proposition \ref{p:key}. In its statement and proof, we use $C$ for a positive constant that may change from line to line, but depends only on $K$, $C_V$, $c_1$, $c_2$, $E$, $E_\infty$, $R_E$, and $m$. We also reuse constants $h_0 \in (0,1]$ and $C_\eta > 0$; they depend only on the same quantities as $C$, except that $h_0$ also depends on $\delta_1 > 0$, while $C_\eta>0$ also depends on $0< \eta < 1$. In particular, $C$ and $h_0$ are independent of $\alpha$, $h$ and $\eta$, and $C_\eta$ is independent of $\alpha$ and $h$.

\begin{proposition} 
\label{p:key}
Fix $K > 0$. Let $V$ as in Theorem~\ref{thm Linfty high d} or~\ref{thm higher dim}, $\sigma$ and $\rho$ be given by \eqref{e:sigma}, $E > E_\infty$ and $0 < \eta < 1$. Then there exist $\tau_0$ as in \eqref{e:IC}, $a_0$ and $M$ as in \eqref{defn a}, radial functions $\mathcal{W}$ and $\Phi$ and their corresponding $w$ and $\varphi$ determined by \eqref{e:IC} and \eqref{e:solve for w and phi}, and constants $C,C_\eta > 0$, $h_0  \in (0,1]$  so that
\begin{equation}
\label{e:keyLower}
A(r)-\frac{K}{2}B(r)\geq \frac{E - E_\infty}{2}w'(r),\qquad\qquad  r \neq a, \, h\in (0,h_0],
\end{equation}
$\varphi_0$ satisfies,
\begin{equation}
\label{e:phiBoundMe}
|\varphi_0(r)|\leq  C  \Big[\frac{1-\alpha}{(1-\tfrac{\eta}{2})(3+\alpha)}\log h^{-1}+ \frac{1}{\eta} \Big], 
\end{equation}
and $w$ satisfies
\begin{gather}
    w(r) \le C_\eta h^{-\frac{4(1-\alpha)}{(2-\eta)(3+\alpha)}} \label{univ bd w},\\
    w'(r) \ge  (r +1)^{-1-\eta} \label{univ lwr bd w prime}, \qquad r \neq a,  \\
    \frac{w(r)^2}{w'(r)} \le C_\eta h^{-\frac{4(1-\alpha)}{(2-\eta)(3+\alpha)}}(1 + r)^{1+\eta}, \qquad r \neq a. \label{w squared over w prime}
\end{gather}
\end{proposition}

\subsection{Small $r$ region}
We start by working with $0 < r  \le a$. Let $\omega\in C_{\text{comp}}^\infty((-3/4,3/4);[0,1])$ with $\omega = 1 $ near $[-1/2,1/2]$. In this region, define $\mathcal{W}$ and $\Phi$ by
\begin{equation}
\label{e:Phase1}
\mathcal{W}=\frac{r(1+\omega(r))}{2},\qquad \Phi=-\frac{1}{r+1+\Phi_1(r)}, \qquad 0 < r \le a.
\end{equation}
where $\Phi_1(s)$ obeying \eqref{Phi 1 conditions} is to be chosen as needed. With these conditions on $\Phi_1$, by Lemma~\ref{l:phi},
\begin{equation}
\label{e:phiBounds}
\frac{\tau_0}{r+1}\leq \varphi_0'(r)\leq \frac{e^{\|\langle s \rangle^{-2}\Phi_1(s)\|_{L^1}}\tau_0}{r+1}, \qquad 0< r \le a.
\end{equation}
 In this region, we work separately on the cases~\eqref{e:Linfty} and~\eqref{Holder short range},

\noindent {\bf{Case~\eqref{e:Linfty}, $\alpha=0$:}}
In this case, we have $V_h=V_h' = 0$, $R_h=V$, and $V_\infty = 0$. Therefore, using \eqref{e:Linfty}, \eqref{e:keyCalc}, and \eqref{e:phiBounds},
\begin{equation} \label{e:simplest A minus KB} \begin{split}
&A-\frac{K}{2}B\\
&\geq w'(E + h^{-2\sigma}(\varphi_0')^2(1+r(1+\omega)\Phi- K (8\tau_0)^{-1} h^{1 + \sigma} r(r+1)(1+\omega)\Phi^2)\\
&\qquad- CK\tau^{-1}_0 h^{-1+\sigma} r(r+1)\langle r\rangle^{-4}m^2)\\
&\geq w'\frac{1}{\tau_0(r+1)^2}( h^{-2\sigma}\tau_0^3(\frac{1+\Phi_1-r\omega}{r+1+\Phi_1})- CKh^{-1+\sigma}m^2) \\
&\qquad + (E - K \tau_0 e^{2\|\langle s \rangle^{-2}\Phi_1(s)\|_{L^1}}  h^{1-\sigma})w', \qquad h > 0.
\end{split}
\end{equation}
So, putting
\begin{equation}
\label{e:g1a}
\Phi_1=\max\Big[\frac{(r+1)m^2+4r\omega-4}{4-m^2},0\Big],
\end{equation}
and then choosing $\tau_0 = \tau_0(C,K,m) \ge 1$ large enough, we obtain, 
\begin{equation} \label{e:key est small r Linfty}
A-\frac{K}{2}B \geq (E-K\tau_0 e^{2\|\langle s \rangle^{-2}\Phi_1(s)\|_{L^1}} h^{1-\sigma})w' \ge \frac{E}{2} w' , \qquad 0 < r \le a, \, h \in (0,h_0],
\end{equation}
for $h_0 = h_0(K, \tau_0, E, m) \in (0, 1]$ small enough. This proves the claimed inequality \eqref{e:keyLower} for\\ $0 < r\le a$.\\
\noindent{\bf{Case~\eqref{Holder short range}, $0\leq\alpha\leq 1$}:}
Recall that $R_{E,V}$ and $\delta_V$ are given by~\eqref{e:R_E} and~\eqref{defn eps0} respectively. Because $R_{E,V}\leq R_E$, and $\delta_V\geq \delta_1$, the first estimate in~\eqref{e:smoothedVbounds} implies
\begin{equation} \label{e:Etilde}
\sup_{\theta\in \US^{n-1}}V_h(r\theta)\leq \frac{E+3V_\infty}{4}\leq \frac{E+3E_\infty}{4}=:\tilde{E},\qquad r\geq  R_E,\, h \in (0,  \delta_1^{1/\rho}].
\end{equation}
Next, let $\psi\in C_{\text{comp}}^\infty((-1,R_E+1);[0,1])$ with $\psi \equiv 1$ on $[0,R_E]$. 
Then, $\sup_{\R^n} V \le C_V$ and~\eqref{e:smoothedVbounds} yield
$$
V_h \leq  C_V \psi(r) +\tilde{E},\qquad  h \in (0,  \delta_1^{1/\rho}].
$$

Using \eqref{e:keyCalc}, \eqref{e:smoothedVbounds}, and \eqref{e:Etilde}, we have the following modified version of the estimate \eqref{e:simplest A minus KB} for $h\in (0,\delta_1^{1/\rho}]$,
\begin{align*}
&A-\frac{K}{2}B\\
& \ge	 w'\Big(E + h^{-2\sigma}(\varphi_0')^2(1+r(1+\omega)\Phi -K(8 \tau_0)^{-1} h^{1 + \sigma} r(r+1)(1+\omega)\Phi^2)  \\
&\qquad-  C_V \psi-\tilde{E}-Ch^{-\rho(1-\alpha)}r\langle r\rangle^{-3}m^2  -CK \tau^{-1}_0 h^{-1+2\rho\alpha+\sigma} r(r+1)\langle r\rangle^{-6}m^{4}\Big)\\
& \ge \frac{w'}{(r+1)^2}\Big( h^{-2\sigma}\tau_0^2(\frac{1+\Phi_1-r\omega}{r+1+\Phi_1})- CK\tau^{-1}_0h^{-1+2\rho \alpha + \sigma} \langle r\rangle^{-2}m^{4} \\
& \qquad- Ch^{-\rho(1-\alpha)} m^2- C_V (R_E+2)^2\psi \Big) +( \tfrac{3}{4}(E - E_\infty) - K \tau_0 e^{2\|\langle s \rangle^{-2}\Phi_1(s)\|_{L^1}} h^{1-\sigma})w'.
\end{align*}
By~\eqref{e:sigma}, we have $0\leq \sigma \leq 1/3$. Using also \eqref{e:g1a}, and choosing $\tau_0 = \tau_0(C, K, C_V, R_E,m) \ge 1$ large enough, we arrive at
\begin{equation} \label{e:key est small r Holder}
\begin{split}
A-\frac{K}{2}B &\geq  ( \tfrac{3}{4}(E - E_\infty) - K \tau_0 e^{2\|\langle s \rangle^{-2}\Phi_1(s)\|_{L^1}} h^{1-\sigma})w' \\
 &\ge \frac{E - E_\infty}{2}w', \qquad 0 < r \le a,\, h \in (0,h_0]
\end{split}
\end{equation}
for $h_0 = h_0(K, \tau_0, E, E_\infty, \delta_1, m) \in (0,1]$ small enough. Here, to see that $h_0$ is independent of $\alpha$, we observe that $1/2 \leq \rho\leq 2/3$ and hence $\delta_1^{1/\rho}\geq \min \{\delta_1^2,\, \delta_1^{3/2} \}$.
\subsection{Large $r$ region}
In the region $r > a$, we handle the cases~\eqref{e:Linfty} and~\eqref{Holder short range} together, taking the worst of the estimates on $R_h$, $V_h$, and $V_h'$. For notational convenience, set $\delta_1 = \rho = 1$ in the case \eqref{e:Linfty}. Then if either \eqref{e:Linfty} or~\eqref{Holder short range} holds, for $h \in (0, \delta_1^{1/\rho}]$,
\begin{equation*}
V_h(r\theta)\leq C_V\psi(r)+\tilde{E},\qquad |V_h'|\leq Ch^{-\rho(1-\alpha)}\langle r\rangle^{-3}m^2(r), \qquad |R_h|\leq C\langle r\rangle^{-2}m(r).
\end{equation*}
Define $\mathcal{W}$ and $\Phi$ for $r > a$ by 
\begin{equation}
\label{e:Phase2}
\mathcal{W}=\frac{(r+1)^{1+\eta}}{2},\qquad \Phi=-\frac{1+\eta}{r+1},\qquad 0 < \eta < 1, \qquad r > a.
\end{equation}
Then, 
$$
\varphi_0'(r)=\varphi_0'(a)e^{\int_a^r \Phi(s)ds}=\varphi_0'(a)\frac{(a+1)^{1+\eta}}{(r+1)^{1+\eta}}, \qquad r > a.$$
Therefore, from \eqref{e:phiBounds},
\begin{equation}
\label{e:phi2}
\frac{ \tau_0 (a+1)^{\eta}}{(r+1)^{1+\eta}} \leq\varphi_0'(r)\leq \frac{ \tau_0 e^{\|\langle s\rangle^{-2}\Phi_1(s)\|_{L^1}}(a+1)^{\eta}}{(r+1)^{1+\eta}}, \qquad r > a.
\end{equation}
We have, using~\eqref{e:keyCalc} once again,
\begin{align*}
A-\frac{K}{2}B
&\geq w'\Big[E+ h^{-2\sigma}(\varphi_0')^2[1-(1+\eta)(r+1)^{\eta}- 8^{-1}K h^{1+\sigma} (r+1)^{1+\eta}\Phi^2 (\varphi'_0)^{-1}]\\
&\qquad -C_V\psi(r) - \tilde{E} - Ch^{-\rho(1-\alpha)}(r+1)^{1+\eta} \langle r\rangle^{-3}m^2(r)\\
&\qquad - CK(r+1)^{1+\eta}h^{-1+\sigma+2\rho\alpha}\langle r\rangle^{-4}m^2(r)(\varphi_0')^{-1}\Big] \\
&\geq -w'\big[C( 1+ \tau^2_0+K\tau^{-1}_0 )h^{-2\sigma}\langle r\rangle^{-2+2\eta}(a+1)^{-\eta} - C_V \psi(r) \big
]\\
&\qquad+(\tfrac{3}{4}(E - E_\infty) - K \tau_0 e^{\|\langle s \rangle^{-2}\Phi_1(s)\|_{L^1}} h^{1-\sigma})w', \qquad h \in (0, \delta_1^{1/\rho}].
\end{align*}
Now, in \eqref{defn a}, fix
\begin{equation}
\label{e:M}
M=\frac{2\sigma}{2-\eta}=\frac{2(1-\alpha)}{(2-\eta)(3+\alpha)}.
\end{equation}
Then taking $a_0 = a_0(C,K, \tau_0, E,E_\infty) \ge 1$ large enough,
\begin{equation} \label{e:key est big r}
\begin{split}
A-\frac{K}{2}B &\geq (\tfrac{3}{4}(E - E_\infty) - C_V\psi + K \tau_0 e^{\|\langle s \rangle^{-2}\Phi_1(s)\|_{L^1}} h^{1-\sigma})w' \\
& \ge \frac{E - E_\infty}{2} w' , \qquad r > a \ge R_E + 1, \, h \in (0,h_0],
\end{split}
\end{equation}
for $h_0 = h_0(K, \tau_0, E, E_\infty, \delta_1, m) \in (0,1]$ small enough. Combining \eqref{e:key est small r Linfty},  \eqref{e:key est small r Holder}, and \eqref{e:key est big r} establishes \eqref{e:keyLower} in either case \eqref{e:Linfty} or \eqref{Holder short range}.

\subsection{Determination of $w$ and $\varphi_0$}
Lemmas~\ref{l:w} and~\ref{l:phi2} complete the proof of Proposition~\ref{p:key}.
\begin{lemma}
\label{l:w}
With $\mathcal{W}$ determined by \eqref{e:Phase1} and \eqref{e:Phase2}, and with initial conditions as in~\eqref{e:IC}, we have
\begin{equation}
\label{e:w}
w=\begin{cases} r&0 < r\leq \frac{1}{2},\\
\frac{1}{2}e^{\int_{1/2}^r \frac{2}{s(1+\omega(s))}ds}&\frac{1}{2}<r\leq a, \\
w(a)e^{\frac{2}{\eta}((a+1)^{-\eta}-(r+1)^{-\eta})}&r> a,
\end{cases}
\end{equation}
and the estimates~\eqref{univ bd w},~\eqref{univ lwr bd w prime}, and~\eqref{w squared over w prime} hold.
\end{lemma}
\begin{proof}
Recalling the definition~\eqref{e:solve for w and phi} of $w$ in terms of $\mathcal{W}$, for $0<\ep<r$,
\begin{equation}
\label{e:weps}
 w(r)=w(\ep)e^{\int_\ep^r\frac{1}{\mathcal{W}(s)}ds}.
\end{equation}
Now, if $0\leq r\leq 1/2$, $\mathcal{W}(r)=r$, therefore, 
$$
w(r)=\frac{w(\ep)}{\ep}r, \qquad 0< \ep \le r \le \frac{1}{2}.
$$
Sending $\ep\to 0^+$ and using $w'(0)=1$, $w(0)=0$, we have 
$$
w(r)=r,\qquad 0\leq r \leq \tfrac{1}{2},
$$
as claimed. The remaining formulae for $w$ in~\eqref{e:w} now follow easily from~\eqref{e:weps} with $\ep=1/2$. 

 To see~\eqref{univ bd w}, note that  $w'= w/\mathcal{W} \geq 0$, so we need only compute $\limsup_{r\to \infty} w(r)$. For this, observe that $ \omega \equiv 0$ on $r\geq 1$. Therefore, for $1\leq r\leq a$,
$$
w(r)=w(1)r^2.
$$
In particular, since 
$$
w(1)\leq\frac{1}{2}e^{\int_{1/2}^r 2s^{-1}ds}=2,$$
 $w(a)=w(1)r^2\leq 2a^2$. Thus (using $a\geq 1$),
\begin{equation*}
\begin{split}
\limsup_{r\to \infty}w(r)=\limsup_{r\to \infty}&w(a)e^{\frac{2}{\eta}((a+1)^{-\eta}-(r+1)^{-\eta})} \\
&\leq 2a^2e^{\frac{2}{\eta}(a+1)^{-\eta}}\leq C_\eta a^2\leq C_\eta h^{-\frac{4(1-\alpha)}{(2-\eta)(3+\alpha)}},
\end{split}
\end{equation*}
as claimed. 

For~\eqref{univ lwr bd w prime}, we first note that $w'(r)=1$ on $0\leq r\leq 1/2$. Then, using $0\leq \mathcal{W}\leq (r+1)^{1+\eta}/2$, we compute
$$
w'(r)=\frac{w(r)}{\mathcal{W}(r)}\geq (r+1)^{-1-\eta}e^{\int_{\frac{1}{2}}^r \frac{1}{\mathcal{W}(s)}}ds\geq (r+1)^{-1-\eta},\qquad r\geq \frac{1}{2}, \, r \neq a.
$$
Finally, to see~\eqref{w squared over w prime}, we observe using~\eqref{univ bd w},
$$
\frac{w^2}{w'}=\mathcal{W}w\leq C_\eta h^{-\frac{4(1-\alpha)}{(2-\eta)(3+\alpha)}} (r+1)^{1+\eta} .
$$
\end{proof}

\begin{lemma}
\label{l:phi2}
With $\Phi$ given by \eqref{e:Phase1} and \eqref{e:Phase2}, and with initial conditions as in~\eqref{e:IC}, we have
\begin{equation}
\label{e:phi}
\varphi_0'(r)=\begin{cases} \tau_0 e^{-\int_0^r \frac{1}{s+1+\Phi_1(s)}ds}&0 < r\leq a,\\
\varphi_0'(a) \frac{(a+1)^{1+\eta}}{(r+1)^{1+\eta}}&r > a.
\end{cases}
\end{equation}
and the estimate~\eqref{e:phiBoundMe} holds.
\end{lemma}
\begin{proof}

The formula~\eqref{e:phi} follows directly from \eqref{e:solve for w and phi}, \eqref{e:Phase1} and \eqref{e:Phase2}. Then, by \eqref{e:phiBounds} and~\eqref{e:phi2},
$$
0\leq \varphi_0'(r)\leq \begin{cases} \frac{\tau_0 e^{\|\langle s\rangle^{-2} \Phi_1(s)\|_{L^1}}}{(r+1)}&0\leq r\leq a\\
  \tau_0 e^{\|\langle s\rangle^{-2} \Phi_1(s)\|_{L^1}} \frac{(a+1)^{\eta}}{(r+1)^{1+\eta}}&r > a.\end{cases}
$$
Using that $a=a_0h^{-M}$, with $M$ as in~\eqref{e:M}, we have, for $h \in (0,1]$,
\begin{equation}
\label{e:phiMax}
\begin{aligned}
|\varphi_0(r)|&\leq \int_0^a\frac{\tau_0 e^{\|\langle s\rangle^{-2} \Phi_1(s)\|_{L^1}}}{s+1}dr+\int_a^\infty \tau_0 e^{\|\langle s\rangle^{-2} \Phi_1(s)\|_{L^1}}\frac{(a+1)^{\eta}}{(s+1)^{1+\eta}}dr\\
&\leq \tau_0 e^{\|\langle s\rangle^{-2} \Phi_1(s)\|_{L^1}}[\log (a+1)+ \frac{1}{\eta}]\\
&=\tau_0 e^{\|\langle s\rangle^{-2} \Phi_1(s)\|_{L^1}}[\log (a_0h^{-\frac{2(1-\alpha)}{(2-\eta)(3+\alpha)}}+1)+ \frac{1}{\eta}]\\
&\leq\tau_0 e^{\|\langle s\rangle^{-2} \Phi_1(s)\|_{L^1}}\Big[\frac{1-\alpha}{(1-\tfrac{\eta}{2})(3+\alpha)}\log h^{-1}+\log (a_0+1)+ \frac{1}{\eta}\Big].
\end{aligned}
\end{equation}
\end{proof}


\section{The one dimensional case}
The key feature we exploit in the one dimensional case is the disappearance of the term involving the operator $\Lambda$ (see~\eqref{Lambda positive}). This removes the requirement that $\mathcal{W}\geq r/2$, allowing \emph{much} more flexibility in the choice of weight function (see \eqref{e:choice1d} below). 

In one dimension we are also able to simplify the approximation of the potential. For $V$ obeying \eqref{Holder condition oned}, and $\chi$ satisfying \eqref{e:chi}, we take
\begin{equation*}
V_h(x) \defeq \int_{-\infty}^\infty V(x + hy) \chi(y)dy.
\end{equation*}
We again define $R_h:=V-V_h$. The following lemma, whose easy proof we omit, gives bounds on $V_h$, $V'_h$ and $R_h$ in one dimension.
\begin{lemma}
Suppose $V$ satisfies the assumptions of Theorem~\ref{thm oned}. Then there exists $C_\chi> 0$ depending only on $\chi$ so that, for all $h\in (0,\delta_{0,V}]$,
\begin{gather}
V_h(x) \le \sup_{y\in[x,x+h]} V(y), \label{sup V h oned} \\
|V_h'(x)|\le C_\chi c_0 h^{-1}m_0(|x|),\label{up bd V h prime oned}\\
|R_h(x)|\leq c_0hm_0(|x|). \label{up bd R h oned}
\end{gather}
\end{lemma}

Similar to the $n\geq 3$ case, the constants $C >0$ and $h_0 \in (0,1]$ which appear in the ensuing estimates may change from line to line, but depend only on $K, C_V$, $c_0$, $E$, $E_\infty$, $R_E$, $\delta_0$ and $m_0$. The constant $C_\eta >0 $ may also depend on $0 < \eta < 1$. In particular, $C$ and $h_0$ are independent of $h$ and $\eta$, and $C_\eta$ is independent of $h$. 

The main result of this section is
\begin{proposition} \label{p:key1d}
Fix $K > 0$ and let $V$ satisfy the assumptions of Theorem~\ref{thm oned}. Let $E > E_\infty$ and $0 < \eta < 1$. Then there exist functions $\mathcal{W}, \Phi :  \R \to [0,\infty)$, and corresponding functions $w$ and $\varphi_0$ determined by and \eqref{e:solve for w and phi}, along with $C, \, C_\eta >0$ and $h_0 \in (0,1]$  such that
\begin{equation}
\label{e:keyLower1}
A(x)-\frac{K}{2}B(x)\geq \frac{E - E_\infty}{2}w'(x), \qquad h \in (0,h_0],
\end{equation}
and
\begin{equation}\label{e:phiBound1}
|\varphi(x)|\leq C ,
\end{equation}
and $w$ satisfies,
\begin{gather}
    w(x) \le 1, \label{univ bd w 1d}\\
      w'(x) \ge   C_\eta e^{-C/h} (|x| +1)^{-1-\eta}, \label{univ lwr bd w prime 1d}  \\
    \frac{w(x)^2}{w'(x)} \le C_\eta (|x| + 1)^{1+\eta}. \label{w squared over w prime 1d}
\end{gather}
\end{proposition}

\begin{proof}
We assume without loss of generality that $m_0(|x|)\geq (1+|x|[\log (|x|+1)]^2)^{-1}$. Then, put
\begin{equation}
\label{e:choice1d}
\Phi=-\frac{2}{|x|+1},\qquad \mathcal{W}=\frac{\delta h}{m_0}.
\end{equation}
for $\delta>0$ to be chosen later. We replace the initial conditions~\eqref{e:IC} with  
$$
w(0)=e^{-\frac{1}{\delta h}\int_0^\infty m_0(s)ds},\qquad \varphi(0)=0,\qquad \varphi'(0)=\tau_0 \ge 1,
$$
where we fix $\tau_0$ below. We find
$$
\varphi'= \frac{\tau_0}{(|x|+1)^{2}},\qquad w= e^{-\frac{1}{\delta h}\int_{|x|}^\infty m_0(s)ds}.
$$
Recall from~\eqref{e:keyCalc} that
\begin{equation} \label{e:penult est oned}
\begin{split}
A-\frac{K}{2}B&\geq w' (E + (\varphi')^2(1+2\mathcal{W}\Phi-K\mathcal{W}\Phi^2\min(\mathcal{W},\tfrac{h}{4\varphi'})) \\
&-V_h-\mathcal{W}V_h' -K\mathcal{W}h^{-2}|R_h|^2\min(\mathcal{W},\tfrac{h}{4\varphi'})).
\end{split}
\end{equation}
Let $\psi\in C_{\text{comp}}^\infty (\mathbb{R};[0,1])$ with $\psi\equiv 1$ on $|x|\leq R_E$ and $\supp \psi \subseteq (-R_E-1, R_E+1)$. Then, by \eqref{sup V h oned}, 
$$
V_h\leq \frac{E+3V_\infty}{4}\leq \frac{E+3E_\infty}{4} ,\qquad   |x|\geq R_E\geq R_{E,V}.
$$
 Combining this with \eqref{up bd V h prime oned}, \eqref{up bd R h oned}, the choice of $\Phi$ and $\mathcal{W}$ in~\eqref{e:choice1d}, and \eqref{e:penult est oned}, we have
\begin{align*}
A-\frac{K}{2}B&\geq w'(E + \tau_0^2(|x|+1)^{-4}(1-4h\delta m_0^{-1}(|x|+1)^{-1}-K \tau_0^{-1} h^2\delta^2 m_0^{-2}(|x|+1)^{-2})\\
&\qquad\qquad- C_V \psi-\tfrac{E+3E_\infty}{4}-C \delta-CK \tau^{-1}_0 \delta^2),
\end{align*}
for $h \in (0, \delta_0]$. First taking $\tau_0=\sqrt{ \max(C_V,1)}(R_E+2)^4$, and then taking $\delta>0$ small enough (depending on $C$, $K$,  $E$, $E_\infty$, $\tau_0$, and $m_0$), we obtain
$$
A-\frac{K}{2}B\geq \frac{E-E_\infty}{2}w', \qquad h \in (0, \delta_0].
$$
To obtain the estimates~\eqref{e:phiBound1},~\eqref{univ bd w 1d},~\eqref{univ lwr bd w prime 1d}, and~\eqref{w squared over w prime 1d}, observe 
$$
\varphi =\tau_0\operatorname{sgn} (x)\left(1-\frac{1}{|x|+1} \right),
$$
and
$$
 w'=\frac{m_0(|x|)}{\delta h} w(x),
$$
and note that $m_0(|x|)\geq C_\eta(|x|+1)^{-1- \eta}$.\\
\end{proof}

\section{Carleman estimates} \label{Carleman estimates}
Our goal in this section is to prove the Carleman estimates needed to establish \eqref{e:mainResult Linfty},  \eqref{e:mainResult} and  \eqref{resolv est oned}. As above, we use $C>0$ to denote a constant that may change from line to line, but depends only $\sup V$, $c_1$, $c_2$, $E$, $E_\infty$ $R_E$ and $m$ ($n \ge 3$) or $\sup V$, $c_0$, $E$, $E_\infty$, $R_E$, and $m_0$  ($n =1$). Besides depending on the same quantities as $C$ does, $h_0\in (0,1]$ depends only on $\delta_1$ ($n \ge 3$) or $\delta_0$ ($n = 1$), and $C_\eta >0$ depends only on $0 < \eta < 1$. So in particular, $C, C_\eta$, and $h_0$ are independent of $\alpha$, $h$ and $\ep \ge 0$.
\begin{lemma} \label{Carleman lemma}
Let $0<\eta< 1$ and suppose that the assumptions of one of Theorem~\ref{thm Linfty high d},~\ref{thm higher dim}, or~\ref{thm oned} hold. Then with $\varphi$ and $w$ and $h_0 \in (0,1]$ as in the statement of Proposition~\ref{p:key} and~\ref{p:key1d} respectively in $n\geq 3$ and $n=1$, we have
\begin{equation} \label{Carleman est}
\|\langle x \rangle^{-\frac{1+ \eta}{2}} e^{\varphi/h} v \|^2_{L^2} \leq
 C_\eta e^{C/h} \|\langle x \rangle^{\frac{1 + \eta}{2}}e^{\varphi/h}(P(h) - E \pm i\varepsilon)v \|^2_{L^2} 
+  C_\eta e^{C/h}  \varepsilon \| e^{\varphi/h}v \|^2_{L^2}.
\end{equation}
for all  $\varepsilon \ge 0$, $ h \in (0, h_0]$, and  $v \in C_{\emph{comp}}^\infty(\mathbb{R}^n)$.
\end{lemma}
\noindent {\bf{Remark:}} Throughout the proof of Lemma \ref{Carleman lemma}, we abuse notation slightly. In dimension $n\geq 3$, we put $\|u(r)\|=\|u(r, \cdot)\|_{L^2(\US_\theta^{n-1})}$, while we put $\|u(x)\|=|u(x)|$ when $n=1$. When $n \ge 3$, $\int_{r,\theta}$ denotes the integral over $(0,\infty) \times \US^{n-1}$ with respect to the measure $dr d\theta$, while $\int_{r,\theta}$ denotes $\int_\R dx$ when $n=1$. 
\smallskip
\begin{proof}
Since $\langle x \rangle^{-(1+ \eta)/2} \le 1$, without loss of generality, we may assume $0 \le \varepsilon \le 1$.

The proof begins from~\eqref{e:lowerDerivative}.
Then, applying~\eqref{e:keyLower} or~\eqref{e:keyLower1}, it follows that for $h \in (0,h_0]$,
\begin{equation} \label{lower bound wF prime}
   w' F + w F' \ge -\frac{3 w^2}{h^2w'} \| P^{\pm}_\varphi(h)u \|^2 \mp  2\varepsilon w \imag \langle u,u'\rangle + \frac{1}{3}w'\|hu' \|^2 +\frac{E-E_\infty}{2} w' \|u\|^2.
\end{equation}

Now we integrate both sides of \eqref{lower bound wF prime}. For $n\geq 3$, we integrate $\int^\infty_0dr$ and use\\ $wF, \, (wF)' \in L^1((0,\infty);dr)$, and $wF(0) = wF(\infty) = 0$, hence $\int_{0}^\infty (wF)'dr = 0$. In dimension $n=1$, we instead integrate $\int_\R dx$ and observe that $\int_{\R} (wF)'dx=0$. Using also \eqref{univ bd w}, \eqref{univ lwr bd w prime} and \eqref{w squared over w prime} when $n \ge 3$, or  \eqref{univ bd w 1d}, \eqref{univ lwr bd w prime 1d} and \eqref{w squared over w prime 1d} when $n =1$, yields, for $h \in (0,h_0]$,
\begin{equation} \label{penult est}
    \int_{r,\theta} (r + 1)^{-1-\eta}\left(|u|^2 +|hu'|^2 \right) \leq C_\eta e^{C/h} \int_{r,\theta} (1 + r)^{1 + \eta}|P^{\pm}_\varphi(h)u|^2  + \varepsilon C_\eta e^{C/h} \int_{r, \theta} |u|^2 + |hu'|^2.
\end{equation}

Moreover,
\begin{equation} \label{rewrite Pu bar u}
    \begin{split}
        \real \int_{r,\theta}(P^{\pm}_\varphi u)\overline{u} &= \int_{r,\theta} |hu'|^2  + \real \int_{r,\theta} 2 h\varphi' u' \overline{u} + \int_{r, \theta} (h^2 \Lambda u)u \\
        &+ \int_{r,\theta} h\varphi''|u|^2 + \int_{r,\theta} \left(V + E - (\varphi')^2 \right)
        |u|^2,\\
    \end{split}
\end{equation}
and 
\begin{equation} \label{int by parts}
    \int_{r, \theta} h \varphi''|u|^2 = - \real \int_{r,\theta} 2 \varphi'h u' \overline{u}. 
\end{equation}
These two identities, together with the facts that  $\Lambda \ge 0$ and $|V + E - (\varphi')^2| \le e^{C/h}$ for $h \in (0,1]$, imply,
\begin{equation} \label{handle deriv u term}
\begin{split}
    \int_{r, \theta} &|hu'|^2 \le e^{C/h} \int_{r,\theta} |u|^2 \\
    &+ \frac{\gamma}{2} \int_{r, \theta}(r +1)^{-1-\eta} |u|^2 + \frac{1}{2\gamma} \int_{r, \theta}(r +1)^{1 + \eta} |P^{\pm}_\varphi(u)|^2, \qquad h \in (0,1], \, \gamma > 0.
    \end{split}
\end{equation}

To finish, we substitute \eqref{handle deriv u term} into the right side of \eqref{penult est}, recall $0 \le \varepsilon \le 1$, and then choose $\gamma > 0$ small enough (depending on $h$ but independent of $\varepsilon$), to get
\begin{equation} \label{final est}
\begin{split}
    \int_{r,\theta} &(r + 1)^{-1-\eta}(|u|^2 + |hu'|^2) \leq \\
    & C_\eta e^{C/h} \int_{r,\theta} (1 + r)^{1 + \eta}|P^{\pm}_\varphi(h)u|^2  + \varepsilon C_\eta e^{C/h}  \int_{r, \theta} |u|^2, \qquad h \in (0, h_0].
    \end{split}
\end{equation}
Since 
\begin{equation*} 
2^{-\frac{1 + \eta}{2}} \le \left( \frac{\langle r \rangle}{r+1} \right)^{1 + \eta},
\end{equation*}
  \eqref{Carleman est} is now an easy consequence of \eqref{final est}.\\
\end{proof}

\section{Resolvent estimates}
In this section, we deduce the resolvent estimates in Theorems \ref{thm Linfty high d}, \ref{thm higher dim} and \ref{thm oned} from the Carleman estimate \eqref{Carleman est}. This same argument has been presented before, see, e.g., \cite{da14, sh17, sh19, vo19a, vo19b}.  But we include it here for the reader's convenience and for the sake of completeness.

The constants $C, \,h_0$, and  $C_\eta $ continue to have the same dependencies as in Section \ref{Carleman estimates}.
\begin{proof}[Proof of Theorems~\ref{thm Linfty high d},~\ref{thm higher dim} and~\ref{thm oned}]
Since increasing $s$ in \eqref{defn g} decreases the resolvent norm, to prove \eqref{e:mainResult Linfty}, \eqref{e:mainResult} and \eqref{resolv est oned}, we may assume without loss of generality that $0 < 2s -1 < 1$. 

Fix $\eta = 2s-1$. When $n \ge 3$, let $\sigma = \sigma_\alpha$ be as in \eqref{e:sigma}. Let  $\varphi$, $w$, and $h_0 \in (0, 1]$ be as in Proposition~\ref{p:key} ($n\geq 3$) or as in Proposition~\ref{p:key1d} $(n=1)$. Then, Lemma \ref{Carleman lemma} holds. Put $C_\varphi = C_\varphi(h) \defeq 2 \max \varphi$. By \eqref{Carleman est}, for some $C, C_s = C_\eta >0$,
\begin{equation} \label{mult through by exp}
e^{-C_\varphi/h}\|\langle x \rangle^{-s} v \|^2_{L^2} \leq
 C_s e^{C/h} \|\langle x \rangle^{s} (P(h) - E \pm i\varepsilon)v \|^2_{L^2} 
+  \varepsilon C_s e^{C/h} \| v \|^2_{L^2},
\end{equation}
for all $v \in C^\infty_{\text{comp}}(\R^n)$, $\varepsilon \ge 0$, and $h \in (0, h_0]$. Moreover, for any $\gamma > 0$,
\begin{equation} \label{epsilon v}
\begin{split}
2\varepsilon \| v \|^2_{L^2} &= -2 \imag\langle (P(h) - E \pm i\varepsilon)v, v \rangle_{L^2} 
\\& \le  \gamma^{-1}\|\langle x \rangle^{s}(P(h) - E \pm i \varepsilon)v \|^2_{L^2} 
+ \gamma\|\langle x \rangle^{-s} v\|^2_{L^2}.  
\end{split}
\end{equation}
 Setting $\gamma = C_s^{-1} e^{-(C+C_\varphi)/ h} $, and using \eqref{epsilon v} to estimate $\varepsilon \| v \|^2_{L^2}$ from above in \eqref{mult through by exp}, we absorb the $\| \langle x\rangle^{-s} v\|_{L^2}$ term that now appears on the right of \eqref{mult through by exp} into the left side. Multiplying through by $2e^{C_\varphi/h}$, and applying \eqref{e:phiBoundMe} $(n\geq 3)$ we arrive at 
\begin{equation} \label{penult}
 \|\langle x \rangle^{-s} v \|_{L^2}^2 \leq C_s e^{h^{-1-\sigma_\alpha}\left(\frac{C\sigma_\alpha}{3 - 2s}\log(h^{-1}) + C_s \right)} \|\langle x \rangle^s(P(h) - E \pm i \varepsilon)v \|_{L^2}^2, \qquad \varepsilon \ge 0, \, h \in (0,h_0].
\end{equation}
In the case $(n=1)$, we apply instead~\eqref{e:phiBound1} to obtain
\begin{equation} \label{penult1}
 \|\langle x \rangle^{-s} v \|_{L^2}^2 \leq C_s e^{Ch^{-1}} \|\langle x \rangle^s(P(h) - E \pm i \varepsilon)v \|_{L^2}^2, \qquad \varepsilon \ge 0, \, h \in (0,h_0].
\end{equation}

The final task is to use \eqref{penult} and~\eqref{penult1} to obtain the corresponding resolvent estimates to show
\begin{equation} \label{ult}
\begin{aligned}
\|\langle x \rangle^{-s}(&P(h)-E \pm i\varepsilon)^{-1} \langle x \rangle^{-s} f \|^2_{L^2}\\
&\leq C_s e^{h^{-1-\sigma_\alpha}\left(\frac{C\sigma_\alpha}{3-2s} \log(h^{-1}) + C_s \right)} \|f \|_{L^2}^2, & \varepsilon > 0, \, h \in (0,h_0], \, f \in L^2,& \qquad (n\geq 3)\\{}\\
\|\langle x \rangle^{-s}(&P(h)-E \pm i\varepsilon)^{-1} \langle x \rangle^{-s} f \|^2_{L^2}\\ &\leq C_s e^{C h^{-1}} \|f \|_{L^2}^2, & \varepsilon > 0, \, h \in (0,h_0], \, f \in L^2, &\qquad (n=1)
\end{aligned}
\end{equation}
    from which Theorems~\ref{thm Linfty high d},~\ref{thm higher dim} and~\ref{thm oned} follow. To establish \eqref{ult}, we prove a simple Sobolev space estimate  and then apply a density argument that relies on \eqref{penult}. 

The operator
\begin{equation*}
[P(h), \langle x \rangle^s]\langle x \rangle^{-s} = \left(-h^2 \Delta \langle x \rangle^s - 2h^2 (\nabla \langle x \rangle^s) \cdot \nabla \right) \langle x \rangle^{-s}
\end{equation*}
is bounded $H^2 \to L^2$. So, for $v \in H^2$ such that $\langle x \rangle^s v \in H^2$,
 \begin{equation}\label{Ceph}
 \begin{split}
\|\langle x \rangle^{s}(P(h)-E \pm i \varepsilon)v\|_{L^2}  &\le \|(P(h)-E \pm i \varepsilon)\langle x \rangle^{s} v \|_{L^2} +  \|[P(h),\langle x \rangle^{s}]\langle x \rangle^{-s}\langle x \rangle^{s}v \|_{L^2}
\\& \le C_{\varepsilon, h} \| \langle x \rangle^{s}v \|_{H^2},
\end{split} 
\end{equation}
for some constant $C_{\varepsilon, h}>0$ depending on $\varepsilon$ and $h$.

Given $f \in L^2$, the function $\langle x \rangle^{s}(P(h)-E\pm i\varepsilon)^{-1}\langle x \rangle^{-s} f \in H^2$ because 
\begin{equation*}
\begin{split}
\langle x \rangle^{s}&(P(h)-E\pm i\varepsilon)^{-1}\langle x \rangle^{-s} f = (P(h) -E \pm i\varepsilon)^{-1} f + [\langle x \rangle^{s}, (P(h) - E \pm i\varepsilon)^{-1}] \langle x \rangle^{-s}f  
\\& =  (P(h)-E \pm i\varepsilon)^{-1} f + (P(h)-E\pm i\varepsilon)^{-1} [P(h),\langle x \rangle^{s}]  (P(h)- E\pm i\varepsilon)^{-1} \langle x \rangle^{-s}f.
\end{split}
\end{equation*}
Now, choose a sequence $v_k \in C_{\text{comp}}^\infty$ such that $ v_k \to  \langle x \rangle^{s}(P(h)-E \pm i\varepsilon)^{-1}\langle x \rangle^{-s} f$ in $H^2$. Define \\ $\tilde{v}_k \defeq \langle x \rangle^{-s}v_k$. Then, as $k \to \infty$,
\begin{equation*}
\begin{split}
\| \langle x \rangle^{-s} \tilde{v}_k - \langle x \rangle^{-s} (&P(h)-E\pm i \varepsilon)^{-1}\langle x \rangle^{-s}f \|_{L^2}  \\
&\le \| v_k - \langle x \rangle^{s} (P(h)-E\pm i \varepsilon)^{-1}\langle x \rangle^{-s}f \|_{H^2} \to 0.
\end{split}
\end{equation*}
Also, applying \eqref{Ceph},
\begin{equation*}
\|\langle x \rangle^{s}(P(h)-E\pm i \varepsilon)\tilde v_k - f\|_{L^2} \le C_{\varepsilon,h} \|v_k - \langle x \rangle^{s} (P(h)-E \pm i \varepsilon)^{-1} \langle x \rangle^{-s} f \|_{H^2} \to 0.
\end{equation*} 
We then achieve \eqref{ult} by replacing $v$ by $\tilde{v}_k$ in \eqref{penult} and sending $k \to \infty$.\\
\end{proof}


\begin{thebibliography}{0}


\bibitem[Bu98]{bu98}
 N. Burq. D\'ecroissance de l'\'energie locale de l'\'equation des ondes pour le probl\`eme ext\'erieur et absence de r\'esonance au voisinage du r\'eel. \textit{Acta Math.}, 180(1) (1998), 1--29

\bibitem[Bu02]{bu02} N. Burq. Lower bounds for shape resonances widths of long range Schr\"odinger operators. \textit{Amer. J. Math.}, 124(4) (2002), 677--735

\bibitem[CaVo02]{cavo} F. Cardoso and G. Vodev.
\newblock Uniform estimates of the resolvent of the {L}aplace-{B}eltrami
  operator on infinite volume {R}iemannian manifolds {II}.
\newblock {\em Ann. Henri Poincar\'e}, 3(4) (2002), 673--691

 \bibitem[DDZ15]{ddz15} K. Datchev, S. Dyatlov and M. Zworski.  Resonances and lower resolvent bounds. \textit{J. Spectr. Theory}, 5(3) (2015), 599--615

\bibitem[Da14]{da14}
K. Datchev. Quantitative limiting absorption principle in the semiclassical limit. \textit{Geom.  Funct. Anal.}, 24(3) (2014), 740--747 



\bibitem[DaSh19]{dash19}
K. Datchev and J. Shapiro. Semiclassical estimates for scattering on the real line, to appear in \textit{Comm. Math. Phys.} arXiv 1903.02743





\bibitem[KlVo19]{klvo19} F. Klopp and M. Vogel. Semiclassical resolvent estimate for bounded potentials. \textit{Pure Appl. Anal.} (1)1 (2019), 1--25 








\bibitem[Sh19]{sh19}
J. Shapiro. Semiclassical resolvent bounds in dimension two. \textit{Proc. Amer. Math. Soc.} 147(5) (2019), 1999--2008 


\bibitem[Sh17]{sh17}
J. Shapiro. Semiclassical resolvent bound for compactly supported $L^\infty$ potentials, to appear in \textit{J. Spectr. Theory.} arXiv 1802.09008




\bibitem[Vo14]{vo14} G. Vodev. Semi-classical resolvent estimates and regions free of resonances. \textit{Math. Nachr.} 287(7) (2014), 825--835

\bibitem[Vo19a]{vo19a} G. Vodev. Semiclassical resolvent estimates for short-range $L^\infty$ potentials. \textit{Pure Appl. Anal.} 1(2) (2019), 207--214

\bibitem[Vo19b]{vo19b} G. Vodev. Semiclassical resolvent estimates for short-range $L^\infty$ potentials. II, to appear in \textit{Asymptot. Anal.} arXiv:1901.01004

\bibitem[Vo20a]{vo19c} G. Vodev. Semiclassical resolvents estimates for $L^\infty$ potentials on Riemannian manifolds. \textit{Ann. Henri Poincar\'e.} 21(2), 437--459

\bibitem[Vo20b]{vo20} G. Vodev. Semiclassical resolvent estimates for H\"older potentials, arXiv:2002.12853.



\end{thebibliography}
\end{document}